\documentclass{agtart_a}
\pdfoutput=1

\title[$\cat(0)$ groups with specified boundary]{CAT(0) groups with specified boundary}
 
\author{Kim Ruane}
\givenname{Kim}
\surname{Ruane}
\address{Department of Mathematics\\
Tufts University\\
Medford, MA 02155\\
USA}
\email{kim.ruane@tufts.edu}
\urladdr{}

\volumenumber{6}
\issuenumber{}
\publicationyear{2006}
\papernumber{23}
\startpage{633}
\endpage{649}

\doi{}
\MR{}
\Zbl{}


\received{13 May 2005}
\revised{}
\accepted{28 February 2006}
\published{4 May 2006}
\publishedonline{4 May 2006}
\proposed{}
\seconded{}
\corresponding{}
\editor{}
\version{}

\arxivreference{}  

\AtBeginDocument{\let\bar\wbar\let\tilde\wtilde\let\hat\what\let\overline\wbar\def\notin{\not\in}}
\def\Fix{\mathrm{Fix}}
\def\Min{\mathrm{Min}}

\def\cat{{\rm CAT}}

\makeatletter
\def\cnewtheorem#1[#2]#3{\newtheorem{#1}{#3}[section]
\expandafter\let\csname c@#1\endcsname\c@thm}

\newtheorem{thm}{Theorem}[section]
\cnewtheorem{proposition}[thm]{Proposition}
\cnewtheorem{lemma}[thm]{Lemma}
\cnewtheorem{claim}[thm]{Claim}
\cnewtheorem{corollary}[thm]{Corollary}
\theoremstyle{definition}
\cnewtheorem{defn}[thm]{Definition}
\cnewtheorem{remark}[thm]{Remark}
\cnewtheorem{example}[thm]{Example}
\cnewtheorem{convention}[thm]{Convention}

\makeautorefname{defn}{Definition}

\makeatother  

\newcommand{\bdry}{\partial_{\infty}}
\newcommand{\bdryT}{\partial_{{\rm Tits}}}

\renewcommand{\emptyset}{\varnothing}

\begin{document}

\begin{abstract}
We specify exactly which groups can act geometrically on CAT(0) spaces
whose visual boundary is homeomorphic to either a circle or a
suspension of a Cantor set.
\end{abstract}

\maketitle

\section{Introduction}

Suppose $X$ is a proper $\cat(0)$ space.  We would like to understand what 
can be said about the geometry of $X$ only knowing the homeomorphism type of 
the visual boundary.  For example, if you specify that $\partial_{\infty}X$ 
is homeomorphic to the circle, then what does this say about $X$?  $X$ 
could be the Euclidean plane, but $X$ could also be $\mathbb H^2$ or 
a Euclidean cone with cone angle greater than $2\pi$.  Such a space can be
obtained by gluing five (or more) quarter planes together along the boundary rays.  
If one further specifies that $X$ admits a geometric group action, then 
one can rule out this last possibility, but the other two possibilities 
remain.

There are two topologies on the boundary $\partial X$ that are used in this paper.  
In general, these boundaries are very different topological objects.  For example, the 
{\it visual} topology on the boundary of $X=\mathbb H^2$ 
gives  $\partial_{\infty}X\cong S^1$ while the Tits topology gives 
$\bdryT X$ an uncountable discrete space.  There is however, 
always a continuous map from $\partial_{{\rm Tits}} X$ to 
$\partial_{\infty} X$ given by the identity map (Bridson and Haeflier 
\cite{BridsonHaefliger}).   
If you know the homeomorphism type of the visual topology and some of the 
structure of the Tits boundary, then you can obtain information about $X$.  However,
it can be difficult to extract any information about the Tits topology only given 
information about the visual topology.  This article will explore the situation 
when the visual boundary is either a circle or a suspension of a Cantor set.  The results 
themselves are not very surprising, but the proof techniques are quite interesting and can 
hopefully be used to analyze other boundaries. 

The first result concerning the circle follows from the Flat Plane Theorem (\fullref{flatpl} here) and work of Gromov, Casson--Jungreis, Tukia, and Gabai.  This result is interesting in its own right but is also used as a crucial step in the proof of the more difficult \fullref{main1}.  Both results use the same method of proof so it is worth reading the easier one first.  

Recall that a $\cat(0)$ space $X$ has {\it local extendability of geodesics} if every geodesic segment in $X$ can be extended to a geodesic line in $X$. 

\medskip
{\bf \fullref{circle}}\qua {\sl Suppose $X$ is a $\cat(0)$ space with local extendability of 
geodesics and $G$ is a group acting geometrically on $X$.  Suppose the 
visual boundary of $X$ is homeomorphic to a circle. Then exactly one of 
the following is true.
\begin{enumerate}
\item $X$ is isometric to a Euclidean plane and $G$ is a Bieberbach group 
(virtually $\mathbb Z\oplus\mathbb Z$).  
\item $X$ is quasi-isometric to $\mathbb H^2$ and $G$ is topologically 
conjugate to a fuchsian group.
\end{enumerate}}

\medskip
The statement of the second result is along the same lines only we assume $\partial_{\infty} X$ is 
homeomorphic to $\Sigma C$ where $C$ denotes a Cantor set, however this does not follow from the work cited above.   The proof of this theorem requires a careful analysis of the interplay between the visual and Tits topologies and the group action.   

\medskip
{\bf \fullref{main1}}\qua {\sl Suppose $G$ acts geometrically on a $\cat(0)$ space $X$ and 
suppose $\partial_{\infty}X$ is homeomorphic to $\Sigma C$ where $C$ 
denotes the Cantor set.  Then, the following are true.
\begin{enumerate} 
\item $\partial_{{\rm Tits}} X$ is isometric to the suspension of an 
uncountable discrete space.
\item  There exists a closed, convex, $G$--invariant subset $X'$ of $X$ which splits as $Y\times\mathbb R$.  If $X$ has local extendability of geodesics, then $X=Y\times\mathbb R$.
\item The subspace $Y$ is $\cat(0)$ and has $\bdry Y$ homeomorphic to 
a Cantor set.
\end{enumerate}}
\medskip

And finally, we obtain information about the group $G$ in the previous theorem using 
an algebraic result proven by M. Bridson and relayed to the author here.  The result of Bridson 
is \fullref{algebra} here. 

\medskip
{\bf \fullref{main2}}\qua {\sl Suppose $G$ acts geometrically on a $\cat(0)$ space $X$ with $\bdry X$ homeomorphic to the suspension of a Cantor set.  Then $G$ contains a subgroup $G_0$ of finite index which is isomorphic to $F\times\mathbb Z$ where $F$ is a nonabelian free group.}

\medskip
As a final remark, we mention that a much more difficult question
would be to classify all $\cat(0)$ groups with Sierpinski carpet
boundary.  An example of a $\cat(0)$ group with this boundary is the
fundamental group of the figure eight knot complement (Ruane
\cite{RuaneTrun}).  This group has $\mathbb Z\oplus\mathbb Z$
subgroups so is not word hyperbolic.  However, if $M$ is a compact
hyperbolic three manifold with nonempty totally geodesic boundary then
$G=\pi_1(M)$ is a word hyperbolic group with Sierpinski carpet
boundary.  The following conjecture concerning word hyperbolic groups
with Sierpinski boundary was posed by Kapovich and Kleiner
\cite{KapovichKleiner}.

\medskip
\noindent
{\bf Conjecture}\qua \cite{KapovichKleiner}\qua  Let $G$ be a word hyperbolic group with Sierpinski carpet boundary.  Then $G$ acts discretely, cocompactly, and isometrically on a convex subset of $\mathbb H^3$ with nonempty totally geodesic boundary.

I would especially like to thank the referee for pointing out a gap in the original version of \fullref{key} and for patiently awaiting a corrected version.  I would also like to thank him/her for many other helpful suggestions and comments that have significantly improved this paper.  I would also like to thank Phil Bowers for his helpful discussions and words of encouragement.

\section{$\cat(0)$ preliminaries}

In this section, we assume $X$ is a proper, complete, geodesic metric 
space.  The $\cat(0)$ inequality is a curvature condition on $X$ first 
introduced by Alexandrov in \cite{Alexandrov}.  The idea is to compare geodesic 
triangles in $X$ to triangles of the same size in the Euclidean plane and 
require those in $X$ to be at least as thin as the corresponding triangle 
in $\mathbb E^2$.  The formal definition is given below.  Examples of $\cat(0)$ 
spaces include $\mathbb E^n$ and the universal covers of compact Riemannian 
manifolds of nonpositive curvature. 

\begin{defn}[$\cat(0)$]  Let $(X,d)$ be a proper complete geodesic
metric space. If $\vartriangle abc$ is a geodesic triangle in $X$, 
then we consider $\vartriangle\overline a\overline b\overline c$ in 
$\mathbb E^2$, a triangle with the same side lengths, and call this a 
{\it comparison triangle}.  Then we say $X$ satisfies the $\cat(0)$ 
{\it inequality} if given $\vartriangle abc$ in $X$, then for any 
comparison triangle and any two points $p,q$ on $\vartriangle abc$, the
corresponding points $\overline p,\overline q$ on the comparison 
triangle satisfy $$d(p,q)\leq d(\overline p,\overline q)$$
\end{defn}

\begin{defn}[Asymptotic] Let $X$ be a metric space.  Two geodesic rays\break
$c,c'\co [0,\infty)\to X$ are called {\it asymptotic} if there exists a 
constant $K$ such that\break $d(c(t),c'(t))\le K$ for all $t\ge 0$.  
\end{defn}

Let $(X,d)$ be a $\cat(0)$ space. First, define the boundary, $\partial 
X$ as a point set as follows:

\begin{defn}[Boundary] The {\it boundary} of $X$, denoted $\partial X$, is the set of
equivalence classes of geodesic rays where two rays are equivalent if and 
only if they are asymptotic.  The union $X\cup\partial X$ is denoted by 
$\wwbar X$. 
\end{defn}

\begin{remark} Fix $x_0\in X$ and consider the set $\partial_{x_0} X$ of
geodesic rays in $X$ that begin at $x_0$.  Fixing a basepoint essentially 
picks out a unique representative of an equivalence class of geodesic rays. 
More precisely, by \cite[Proposition 8.2]{BridsonHaefliger}, given a geodesic ray 
$c\co [0,\infty)\to X$ with $c(0)=x_0$ and any other $x\in X$, there is a unique 
geodesic ray $c'\co [0,\infty)\to X$ issuing from $x$ and is asymptotic to $c$.  
Thus one can identify $\partial X$ with $\partial_{x_0} X$ for a given basepoint.
We describe a topology on $\partial_{x_0} X$ below and it follows that the obvious 
bijection $\partial_{x_0} X\to\partial_{x_1} X$ for any other $x_1\in X$ is a
homeomorphism.  Thus we often write $\partial X$ even though we may be 
considering a fixed basepoint.  
\end{remark}

There is a natural neighborhood basis for a point in $\wwbar
X$.  Let $c$ be a geodesic ray emanating from $x_0$ and $r>0,\  \epsilon
>0$. Also, let $S(x_0,r)$ denote the sphere of radius $r$ centered
at $x_0$ with $p_r\co X \to S(x_0,r)$ denoting projection.  Define
$$U(c,r,\epsilon)=\{ x\in\wwbar{X} \vert d(x,x_0)>r,\ d(p_r(x),c(r))<
\epsilon\}$$  
This consists of all points in $\wwbar{X}$ such that when projected
back to $S(x_0,r)$, this projection is not more than $\epsilon$ away from
the intersection of that sphere with $c$.  These sets along with the
metric balls about $x_0$ form a basis for the {\it cone topology} on $\wwbar X$. The 
set $\partial X$ with the subspace topology is often called the 
{\it visual boundary}.  As one expects, the visual boundary of 
$\mathbb E^n$ is $S^{n-1}$ as is the visual boundary of $\mathbb H^n$.  
Thus the visual boundary does not capture the difference between these 
two $\cat(0)$ spaces.  Thus we need the Tits topology to distinguish 
these  types of spaces.   The notation $\partial X$ is used to denote the
visual topology and $\bdryT X$ to denote the Tits boundary.

We first develop a general technique for measuring
angles between points in $\partial X$.  We assume the reader has some
knowledge of how to measure angles in a metric space, but we add the
necessary definitions for completeness.  Alexandrov used the method of
comparison triangles to define the notion of angle between two geodesics
leaving a point $x_0$ in a metric space $X$, \cite{BridsonHaefliger}.  We recall that
definition here. 

\begin{defn}[Angles] Let $c\co [0,a]\to X$ and $c^\prime\co [0,a^\prime]\to
X$ be two geodesics with $c(0)=c^\prime (0)=x_0$.  Given $t\in [0,a],
t^\prime\in [0,a^\prime]$, and let $\alpha^{t,t^\prime}_{c,c^\prime}$
denote the angle in a comparison triangle in Euclidean space at the
vertex corresponding to $x_0$.  The (upper) angle between $c,c^\prime$ at
$x_0$ is defined to be the following number:
$$\angle_{c,c^\prime} := \limsup_{t,t^\prime\to 0}
\alpha^{t,t^\prime}_{c,c^\prime}$$
\end{defn}

Note: The $\limsup$ is used because the limit may not always exist, but in
$\cat(0)$ spaces, the limit does exist and instead of calling it an
``upper'' angle, we call it the angle.  For a proof of this, see \cite{BridsonHaefliger}.

\begin{defn}[Angle metric]\label{angle}  Let $X$ be a $\cat(0)$ space.  Given $x\in X$
and $u,v\in\partial X$, we denote by $\angle_x(u,v)$ the angle between
the unique geodesic rays which issue from $x$ ending at 
$u$ and $v$ respectively.  Then we define the angle between $u$ and
$v$ to be $$\angle(u,v)= \sup_{x\in X}\angle_x(u,v)$$
\end{defn}

One can verify this angle satisfies the triangle inequality
and so $\angle (u,u^\prime)$ defines a metric on $\partial X$ called the 
{\it angular metric}.  

\begin{example} Consider the $\cat(0)$ space $X=\mathbb H^2$.  We know
$\partial X$ can be identified with the unit circle in the disk model.  For any two 
points $p,q$ on the boundary circle, there is a geodesic line in $\mathbb H^2$ joining $p$ and $q$.   Clearly the supremum is attained by taking any other point on the line between $p$ and $q$ so $\angle (p,q)=\pi$.  This means the angle metric on $\partial(\mathbb H^2)$ is discrete.
\end{example}

\begin{defn}[Tits metric] The {\it Tits metric} on $\partial X$, denoted $Td$, is the {\it length}
metric associated to the angular metric.  Thus for $v,w\in\partial X$,
$Td(v,w)$ is defined to be the infimum of the lengths of rectifiable
curves in the angular metric between $v$ and $w$.  If there are no
rectifiable curves joining them, then $Td(v,w)=\infty$.  
\end{defn}

From the example above, we can conclude that for any two points $u,v\in\partial\mathbb H^2$, $Td(u,v)=\infty$.  The basic facts needed here concerning the Tits metric can be found in Ballmann--Gromov--Schroeder
\cite{BGS} for the manifold setting  and in \cite{BridsonHaefliger} for the $\cat(0)$ setting.  A $\cat(1)$ space is defined similarly to a $\cat(0)$ space except the comparison triangles live in the standard $S^2$ rather than $\mathbb E^2$ and the inequality must hold for all triangles with perimeter less than $2\pi$.   Note, that this perimeter restriction guarantees the existence of a comparison triangle in $S^2$.  

\begin{remark}\label{tits} The main facts concerning $\bdryT X$ needed here are as follows.
\begin{enumerate}
\item The Tits boundary of a flat plane in $X$ is a circle that is 
isometrically embedded in $\bdryT X$
\item  If $u,v\in\partial X$ have $Td(u,v)=\infty$, then there is a 
geodesic line in $X$ joining $u$ and $v$ and this line does not bound a flat half 
plane.  
\item $\bdryT X$ is a $\cat(1)$ space.  In particular, if $p,q\in\bdryT X$ have $Td(p,q)<\pi$, then 
there is a unique geodesic between them in $\bdryT X$ and $Td(p,q)=\angle(p,q)$.  

\end{enumerate}
\end{remark}

\begin{remark}  It follows that if $X$ and $Y$ are two $\cat(0)$ spaces, 
then with the product metric, $X\times Y$ is also $\cat(0)$ and 
$\partial_{\infty}(X\times Y)\cong\partial_{\infty} X\star\partial_{\infty} Y$ 
where $\star$ denotes the  spherical join.  The Tits topology is 
isometric to the join of the Tits topologies of the factors.
In this paper the case in which one of the factors is $\mathbb 
R$ is used so we point out the following:  
$\partial(Y\times\mathbb R)\equiv\Sigma(\partial Y)$ where $\Sigma$ denotes 
suspension. 
\end{remark}

We will also need the following facts concerning lines and planes in a proper $\cat(0)$ 
space $X$.  Again, proofs can be found in \cite{BridsonHaefliger}.  

\begin{defn}[Parallel]\label{parallel} Suppose $c,c'\co\mathbb R\to X$ are geodesic lines in $X$.  Then we say $c,c'$ are {\it parallel} in $X$ if there exists a constant $K\ge 0$ with $d(c(t),c'(t))\le K$ for all $t\in\mathbb R$. 
\end{defn}

\begin{thm}\label{lines} Suppose $c\co\mathbb R\to X$ is a geodesic line.  Let $\mathcal{P}(c)$ denote the set of all geodesic lines in $X$ that are parallel to $c$.  Then $\mathcal{P}(c)$ is a closed convex subset of $X$ which splits isometrically as $\mathcal{P}(c)\equiv Y\times\mathbb R$ where the $\mathbb R$ factor is determined by $c$. $Y$ is also a closed, convex subset of $X$.  
\end{thm}

The next result is analogous to the previous result but for any closed
convex subset $C$ of $X$ instead of a line.  This is often referred to
the Sandwich Lemma, \cite[page 182]{BridsonHaefliger}.  Write
$d_C(x)=\inf \{d(x,c)\ |\ c\in C\}$ to denote the distance from a
point $x\in X$ to the set $C$.

\begin{thm}[Sandwich Lemma]\label{flats} Let $C_1,C_2$ be closed convex subspaces of a $\cat(0)$ space $X$.  If the restriction of $d_{C_1}$ to $C_2$ is constant and equal to some number $a$, and the restriction of $d_{C_2}$ to $C_1$ is constant, then the convex hull of $C_1\cup C_2$ is isometric to $C_1\times [0,a]$.  
\end{thm}

We now state the version of the Flat Plane Theorem used in this paper.  
This version is stated  as \cite[Theorem 3.1, page 459 ]{BridsonHaefliger}. 

\begin{thm}[Flat Plane Theorem]\label{flatpl} If a group $\Gamma$ acts properly and cocompactly by isometries on a $\cat(0)$ space $X$, then $\Gamma$ is word hyperbolic if 
and only if $X$ does not contain an isometrically embedded copy of the 
Euclidean plane.  
\end{thm}

\begin{remark} In the case that $\Gamma$ is word hyperbolic, we have 
$\Gamma$ is quasi-isometric to $X$ as metric spaces.  This forces $X$ to 
be $\delta$--hyperbolic as well since hyperbolicity is preserved under 
quasi-isometries.
\end{remark}

We will also need the following result concerning free abelian central 
subgroups.  This can be found as \cite[Theorem 6.12, page 234]{BridsonHaefliger}.  

\begin{thm}\label{central} Let $X$ be a $\cat(0)$ space and let $\Gamma$ be a finitely 
generated group acting by isometries on $X$.  Suppose $\Gamma$ contains a 
central subgroup $A\cong\mathbb Z^n$ that acts faithfully by hyperbolic 
isometries (apart from the identity), then there exists a subgroup of 
finite index $H\subset\Gamma$ which contains $A$ as a direct factor.
\end{thm}  

\section{Geometric results}

One of the important facts used for the proofs in this section is that the identity map $\partial X\to\partial X$ gives a continuous function $\bdryT X\to\bdry X$ between these two topological spaces.  Intuitively, if $p,q\in\partial X$ have small angle in the sense of \fullref{angle}, then if we take rays from a fixed basepoint $x_0$ to $p$ and $q$, these rays must fellow travel for a long time.  This is not necessarily true in reverse as can be seen in $\mathbb H^2$ since $\angle (p,q)=\pi$ for all $p\ne q\in\partial\mathbb H^2$.   The next lemma follows easily from point-set topology, however the proof given here uses this interaction between the two topologies. 

\begin{lemma}  Suppose $X$ is a $\cat(0)$ space and $F$ is an $n$--flat in 
$X$.  Then the identity map from $\bdryT X$ to 
$\bdry X$ restricted to the boundary of $F$ is a 
homeomorphism.  In particular, the image of $\bdryT F$ in 
$\partial_{\infty} X$ is a simple closed, curve.  
\end{lemma}

\begin{proof}  Clearly the restricted map is a continuous bijection from the Tits
circle to the visual circle.  We show the inverse map is continuous.  Pick 
a basepoint $x_0\in F$.  Given a point $p\in\bdryT  F$ and an 
$\alpha >0$ we can find $r$ and $\epsilon$ so that the set 
$U(p,r,\epsilon)\cap\partial_{\infty} F$ consists of points whose angle 
with $p$ is less than $\alpha\slash 2$.  Indeed, use the law of cosines which holds 
in a $\cat(0)$ space. 
\end{proof}

We will also need the following lemmas for both theorems.  A
(complete) $\cat(0)$ space $X$ has {\it local extendability of
geodesics} if every geodesic segment in $X$ can be extended to a
geodesic line in $X$ (compare \cite[Definition
II.5.7]{BridsonHaefliger}).  It is not true in general that a complete
$\cat(0)$ space has this property, however it is known that if $X$ is
proper and the full isometry group of $X$ has a compact fundamental
domain, then $X$ is {\it almost extendable}.  This means there exists
a global constant $r\ge 0$ such that for any pair $a,b\in X$ there
exists a geodesic ray beginning at $a$ and passing within $r$ of $b$.
See Geoghegan and Ontaneda \cite{GO} for details.

The assumption of local extendabililty of geodesics is necessary for making certain conclusions about the space $X$, however for results about $\bdry X$ only it is often unnecessary to have this extra assumption.   Instead, the notion of a quasi-dense set is sufficient for results about $\bdry X$.  A subset $M$ of $X$ is called {\it quasi-dense} if there exists a constant $K>0$ so that each point of $X$ is within $K$ of some point of $M$.  Notice that if a group $\Gamma$ acts cocompactly, then the orbit of a point is a quasi-dense subset.

\begin{lemma} If $M\subset X$ is closed, convex, and
quasi-dense, then $\partial X=\partial M$.
\end{lemma}

\begin{proof} If $p\in\partial X$, then let $c\co[0,\infty)\to X$ be a geodesic ray from $x_0$ whose endpoint in $\wwbar{X}$ is $p$.  For each $n\in\mathbb N$, choose $m_n\in M$ with $d(m_n,c(n))<K$.  The geodesics $[x_0,m_n]$ are all contained in $M$ since $M$ is convex.  The sequence $[x_0,m_n]$ converges to a geodesic ray $c'$ in $M$ with endpoint in $\partial M$.  The endpoint of $c'$ must be $p$ since $c'$ is asymptotic to $c$ by construction, thus $p\in\partial M$ and we are done.
\end{proof}

\begin{corollary}\label{all} Suppose the group $G$ acts geometrically on a $\cat(0)$ 
space $X$ and suppose $Z$ is a closed, convex, $G$--invariant subset of 
$X$.  Then $\partial Z=\partial X$.  Furthermore, if $X$ has local 
extendability of geodesics, then $Z=X$.  
\end{corollary}

\begin{proof} Let $z_0\in Z$.  Then the orbit $G\cdot z_0\subset Z$ is a quasi-dense subset of $X$ which forces $Z$ to be quasi-dense.  The corollary now follows from the previous lemma.
\end{proof}

\begin{lemma}\label{compact}Suppose $X$ is a $\cat(0)$ space with 1--dimensional visual boundary $\bdry X$. Suppose $F$ is a 2--flat in $X$ and let $\mathcal{P}(F)$ be the set of all flats in $X$ that are parallel to $F$.  Then $\mathcal{P}(F)$ is isometric to a product $F\times Y$ with $Y$ compact. 
\end{lemma}

\begin{proof}
By the Sandwich Lemma (\fullref{flats} here), the set $\mathcal{P}(F)$ of all flats in $X$ 
parallel to $F$ isometrically splits as $F\times Y$ for some closed, convex subset $Y$ of $X$. 
We claim that $Y$ must be compact.  Indeed, otherwise, $\partial Y\ne\emptyset$ 
and $\bdryT (F\times Y)$ is isometric to the spherical join of 
$\bdryT F$ and $\bdryT Y$.  If $Y\ne\emptyset$, this 
would force $\bdryT X$ to be 2--dimensional since $\bdryT F$ 
is a circle.  This would force the visual topology to contain a 
2--dimensional subset which is gives a contradiction.
\end{proof}

The following result will provide one of the crucial steps in \fullref{main1}.  
   
\begin{thm}\label{circle} Suppose $X$ is a $\cat(0)$ space with local extendability of 
geodesics and $G$ is a group acting geometrically on $X$.  Suppose the 
visual boundary of $X$ is homeomorphic to a circle. Then exactly one of 
the following is true.
\begin{enumerate}
\item $X$ is isometric to a Euclidean plane and $G$ is a Bieberbach group 
(virtually $\mathbb Z\times\mathbb Z$).  
\item $X$ is quasi-isometric to $\mathbb H^2$ and $G$ is topologically 
conjugate to a fuchsian group.
\end{enumerate}
\end{thm}

\begin{proof}  We know by the Flat Plane Theorem that $X$ is either 
$\delta$--hyperbolic or $X$ contains a 2--flat.  Suppose $X$ contains a 2--flat 
-- ie, an isometric embedding of $f\co\mathbb E^2\to X$.   Denote by $F$, 
the image of $f$.  Let $g\in G$ and consider the flat $g\cdot F$.  Each of 
$F$ and $g\cdot F$ contribute a circle to the Tits boundary of $X$.  There 
are three possibilities -- the circles are the same, they intersect but are 
unequal, or they are disjoint.  We claim the last two cannot happen.  Indeed, if 
the two circles are disjoint, then they map homeomorphically to disjoint 
circles in the visual topology. This obviously cannot happen since 
$\partial_{\infty} X$ is homeomorphic to a single circle.  If the two 
circles intersect but are not the same, then again, this would happen in 
the visual topology which is absurd.  Thus $F$ and $g\cdot F$ have the 
same boundary circle.  This implies that $g\cdot F$ is parallel to $F$.  Consider the set
$\mathcal{ P}(F)$ of all flats in $X$ that are parallel to $F$.  By \fullref{compact}, $\mathcal{P}(F)=F\times Y$ where $Y$ is compact.  Since the chosen $g$ was arbitrary, 
$F\times Y$ is a closed, convex, $G$--invariant subset of $X$.  By \fullref{all}, we have $X=F\times Y$.  With the assumption of local extendability, $Y$ must be a single point and we are in the 
first case of the theorem.  

Suppose $X$ is $\delta$--negatively curved.  Then the group $G$ is word 
hyperbolic and $\partial G$ (the Gromov boundary of the group) is the same 
as $\bdry X$, \cite{BridsonHaefliger}.  We also know $G$ acts as a 
convergence group on $\bdry X$ (Gromov \cite{Gromov}).   But now the work of 
Casson--Jungreis \cite{CassJung}, Tukia \cite{Tukia}, Gabai \cite{Gabai} implies that $G$ is topologically conjugate to a fuchsian group.  This also gives $G$ quasi-isometric to $\mathbb H^2$ which, in turn, gives $X$ is quasi-isometric to $\mathbb H^2$ as needed 
for case two of the theorem.
\end{proof}

\begin{remark}  Note that in case two of the theorem, we do not get that 
$X$ is isometric to the hyperbolic plane.  The following example was 
explained to me by Phil Bowers.  If you take the constant seven degree 
triangulation of the plane, make all edge lengths one, and glue in 
equilateral triangles, you will obtain a $\cat(0)$ space  $X$ which is 
quasi-isometric to $\mathbb H^2$ but not isometric to it.  Indeed, there 
are flat triangles in $X$.  Furthermore, the (2,3,7) triangle group acts 
in the obvious way as a geometric group action on $X$.  Thus we leave the 
statement as it is in the theorem and cannot hope for more.
\end{remark}

\begin{thm}\label{main1}  Suppose $G$ acts geometrically on a $\cat(0)$ space $X$ and 
suppose $\partial_{\infty}X$ is homeomorphic to $\Sigma C$ where $C$ 
denotes the Cantor set.  Then, the following are true.
\begin{enumerate} 
\item $\bdryT X$ is isometric to the suspension of an 
uncountable discrete space.
\item  There exists a closed, convex, $G$ invariant subset $X'$ of $X$ which splits as $Y\times\mathbb R$. If $X$ has local extendability of geodesics, then $X=Y\times\mathbb R$. 
\item The subspace $Y$ is $\cat(0)$ and has $\bdry Y$ homeomorphic to 
a Cantor set.
\end{enumerate}
\end{thm}

\begin{proof}  By the Flat Plane Theorem, we know $X$ is either 
$\delta$--negatively curved or else $X$ contains a flat plane.  If $X$ 
negatively curved, then $G$ is word hyperbolic.  Since 
$\bdry X$ is connected, $G$ is one-ended.  But  $\bdry X$ is a suspension of a 
Cantor set and is thus not locally connected.  This contradicts the work 
of Swarup \cite{Swarup} on boundaries of one-ended hyperbolic groups.  

Thus $X$ must contain a flat plane, $F$.  Since $F$ is a closed, convex subset 
of $X$, we know $\bdry F\cong S^1$ embeds in $\partial_{\infty} 
X$.  Also $\bdryT F\equiv S^1$ embeds in $\bdryT X$.   ($\cong$ means ``homeomorphic to" while 
$\equiv$ means ``isometric to"). 

Note that any circle in $\bdry X$ must go through the suspension points of $\bdry X$.  Let's call these suspension points $N$ and $S$. 

We claim $Td(N,S)=\pi$.  We prove this fact in a subsequent lemma.

Thus there is a geodesic line $c\co\mathbb R\to X$ 
with image in $F$, with endpoints $N$ and $S$ in $\bdry X$.  
Consider $\mathcal{P}(c)$ in $X$, the set of all geodesic lines in $X$ that are 
parallel to the line $c$.  

We know by \fullref{lines} that $P(c)\equiv Y\times\mathbb R$ is a closed convex 
subset of $X$.  Furthermore, $\bdryT \mathcal{P}(c)\equiv\Sigma(\bdryT Y)$.  The suspension points of this are $N$ and $S$, the suspension points of $\bdry X$.  

Given $g\in G$, we have $g\cdot F$ is another flat in $X$ with $\bdryT (g\cdot F)$ an $S^1$ in 
$\bdryT X$.  There are three possible configurations. 

\begin{enumerate}
\item $\bdryT F$ and $\bdryT (g\cdot F)$ are disjoint.
\item $\bdryT F=\bdry (g\cdot F)$.
\item $\bdryT F\cap\bdryT (g\cdot F)\ne\emptyset$ but are unequal.
\end{enumerate}
If the two circles are disjoint, then they would map to disjoint circles in $\bdry X$ under the identity map but this is impossible since any pair of circles in $\bdry X$ intersect.  Thus case (1) cannot happen. 

If $\bdryT F=\bdryT (g\cdot F)$, then the flats $F$ and $g\cdot F$ are parallel in $X$.  We claim there must be some $g\in G$ for which this does not happen.  If this were true for all elements $g\in G$, then consider the set  $\mathcal{P}(F)\equiv Y\times F$ of all flats in $X$ that are parallel to $F$.  $Y$ must be compact by \fullref{compact}.  Now the group $G$ leaves $\mathcal{P}(F)$ invariant and thus acts geometrically on $\mathcal{P}(F)$.  But $\bdry \mathcal{P}(F)$ is a circle so $G$ and $X$ fall into the setting of \fullref{circle}.   

Thus there exists $g\in G$ as in the case (3) above.  In this case, the circles $\bdryT F$ and $\bdryT (g\cdot F)$ must both contain $N$ and $S$ since any two distinct circles in $\bdry X$ share the points $N$ and $S$.   Thus the line $g\cdot c$ in $g\cdot F$ is parallel to the line $c$ in $F$.  

Since the action of $G$ must take flats in $X$ to flats in $X$, we clearly 
have that $G$ leaves $\mathcal{P}(c)$ invariant.  Thus $\mathcal{P}(c)$ is a closed, convex, 
$G$--invariant subset of $X$.  By \fullref{all} $\mathcal{P}(c)$ is quasi-dense in $X$ and 
$\partial X=\partial \mathcal{P}(c)=\Sigma\partial Y$.  If $X$ has local extendability of geodesics, 
we indeed have $\mathcal{P}(c)=X$ by the same lemma.  We have now shown the second item
in the theorem is true.  We also know that $\bdry X=\bdry\mathcal{P}(c)\cong\Sigma(\bdry Y)$ where 
$N$ and $S$ are the suspension points. 

Let $z\ne w\in\partial Y$.  Since $\bdryT X$ is isometric to the 
suspension of $\partial Y$, we clearly have $Td(z,w)\le\pi$ because  we can
construct a Tits path from $z$ to $w$ by combining the segments $[z,N]$ and $[N,w]$ (or $S$ instead of $N$).  Each of these segments is length $\pi\slash 2$ and thus we have a path from $z$ to $w$ of length $\pi$.  We claim that $Td(z,w)=\pi$.  

If there were a shorter Tits path between them, then this path would not pass through $N$ or $S$.  
This path would also give a path in the visual topology after mapping via the identity map between the two topologies.  Since $\bdry X$ is the suspension of a Cantor set and $z,w\notin\{N,S\}$, the only paths between them go through $N$ or $S$ unless they are on the same suspension line.  But this is clearly not the case for two points in $\partial Y$.  Thus the Tits metric restricted to $\partial Y$ is discrete and item (1) of the theorem is complete. 

Since $\bdry X\cong\Sigma(\bdry Y)$, each point $p\in\partial X$ which is neither $N$ or $S$ lies on a unique arc between $N$ and $S$ in $\bdry X$.  The assumption that $\bdry X$ is also homeomorphic to $\Sigma C$ (with suspension points $N$ and $S$) implies $\bdry Y$ is homeomorphic to $C$.  Thus all three items in the theorem are true. 
\end{proof}

The proof of \fullref{key} below requires several facts about the action of a hyperbolic isometry on the boundary.   If $G$ is a group acting geometrically on a $\cat(0)$ space $X$ and $g\in G$ of infinite order, then $g$ acts on by translation on a geodesic line in $X$ with two endoints in $\bdry X$.  This $g$ also acts on $\bdry X$ as a homeomorphism with a special fixed point set which we discuss below.   Every $g\in G$ will fix the suspension points of $\bdry X$ simply because of topology.  The key observation is that if two infinite order elements (which do not have a common power) fix the same suspension line between the suspension points of $\bdry X$, then one of them must have the suspension points as its endpoints in which case the lemma will be true.  

Let $g$ be an isometry of a $\cat(0)$ space $X$.  Recall that $\Min(g)=\{x\in X\ |\ d(x,g\cdot x)=|g|\}$ where $|g|=\inf\{d(x,g\cdot x)\ |\ x\in X\}$.  We say $g$ is {\it hyperbolic} if $|g|>0$ and $\Min(g)\ne\emptyset$.  In this case, we have the following facts whose proofs can be found in  Ruane \cite{RuaneDyn}, Theorem 1.2.3 and Theorem 3.2.

\begin{thm}\label{minset} Suppose $g$ is a hyperbolic isometry of a $\cat(0)$ space $X$.  Then
\begin{enumerate}
\item There exists a geodesic line $A_g$ called an {\it axis} on which $g$ acts via translation by $|g|$. The set $\Min(g)$ consists of all points in $X$ which lie on such an axis.  
\item $\Min(g)$ is isometric to $Y\times\mathbb R$ where $Y$ is a closed, convex subset of $X$.  
\item $\bdry \Min(g)$ is homeomorphic to $\Sigma(\bdry Y)$ and $\bdryT \Min(g)$ is isometric to\break $\Sigma (\bdryT Y_g)$ where the suspension points are the endpoints of $A_g$ (and hence any $g$--axis)
\item Suppose $g\in G$ where $G$ acts geometrically on $X$.  Then the centralizer $C(g)$ in $G$ acts geometrically on $\Min(g)$ and $C(g)\slash\langle g\rangle$ acts geometrically on $Y$ (via the projected action).  
\end{enumerate} 
\end{thm}

If $g\in G$ where $G$ acts geometrically on a $\cat(0)$ space $X$,
then any infinite order element is a hyperbolic isometry.  This is the
contents of \cite[Proposition 6.10(2), page 233]{BridsonHaefliger}.

Any isometry $g$ of $X$ extends to a homeomorphism of $\bdry X$ and an
isometry of $\bdryT X$.  We denote both of these maps by
$\overline{g}$ even though this is a slight abuse of notation.  It
will be clear which one of these we are talking about in context.
Recall we use $\partial X$ to denote the point set of the boundary,
$\bdry X$ to denote the visual boundary and $\bdryT X$ to denote the
Tits boundary.  A proof of the following can by found as \cite[Theorem
3.2]{RuaneDyn} and also in \cite{Swenson}.

\begin{thm}\label{dynamics} $\Fix(\overline{g})=\bdry \Min(g)$.   
\end{thm}

Finally, we will need two key results from \cite{Swenson} that make our proof work.  The first is Theorem 11 from that paper. 

\begin{thm}\label{swenson1}If an infinite group $G$ acts geometrically on a $\cat(0)$ space $X$, then $G$ contains an element of infinite order.
\end{thm}

A subgroup $H$ of a group $G$ acting geometrically on a $\cat(0)$ space $X$ is called {\it convex} if there exists a closed, convex $H$--invariant subset $A$ of $X$ on which $H$ acts geometrically.  In this case the limit set $\Lambda H$ is the same as $\bdry A$.  The following is really the key result we use in the proof.  The proof of this fact from Swenson uses the full strength of a geometric group action on a $\cat(0)$ space.  In particular, the cocompactness of the actions are essential.  This is \cite[Theorem 16]{Swenson}. 

\begin{thm}\label{swenson2}Suppose $H$ and $K$ are convex subgroups of a group $G$ acting geometrically on a $\cat(0)$ space $X$.  Then 
\begin{enumerate}
\item $H\cap K$ is convex.
\item $\Lambda (H\cap K)=\Lambda H\cap\Lambda K$.
\end{enumerate}
\end{thm}

Finally, we point out an obvious corollary to the Flat Plane Theorem which will be used in the proof of \fullref{key}. 

\begin{corollary}\label{halfplane} Suppose $G$ acts geometrically on the $\cat(0)$ space $X$.  Then $\bdryT X$ cannot be isometric to a segment of length less than or equal to $\pi$.  
\end{corollary}

\begin{proof}  By \fullref{flatpl}, either $X$ is $\delta$--hyperbolic or $X$ contains a flat plane.  If $X$ were $\delta$--hyperbolic, $\bdryT X$ would be discrete and thus would not contain a segment.  Thus $X$ contains a flat plane which implies $\bdryT X$ contains at least a circle.
\end{proof}

We are now ready to prove the Lemma. 

\begin{lemma}\label{key} The suspension points $N$ and $S$ in \fullref{main1} have $Td(N,S)=\pi$. 
\end{lemma}

\begin{proof} We know $Td(N,S)\le\pi$ since $N$ and $S$ lie in the boundary of a common flat.  Suppose $Td(N,S)=\alpha <\pi$.  Using \fullref{tits} item (3), there is a unique segment in $\bdryT X$ between $N$ and $S$ of length $\alpha$.   Since the entire group $G$ must fix $N$ and $S$, $G$ must fix this segment in $\bdryT X$.  Thus either $N$ and $S$ are each fixed by all of $G$ or they are interchanged.  In the second case, we can pass to a subgroup of index 2 in $G$ which fixes both $N$ and $S$.  Thus we can assume each $g\in G$ fixes both $N$ and $S$. 

Using \fullref{swenson1}, we know there exists $g\in G$ of infinite order.  Consider $\Min(g)=Y_g\times\mathbb R$.  By \fullref{minset}, this a closed, convex (and thus $\cat(0)$)  subset of $X$ on which the subgroup $C(g)$ acts geometrically with $C(g)\slash\langle g\rangle$ acting geometrically on $Y_g$.   Let $\{g^{\pm\infty}\}$ denote the endpoints in $\bdry X$ of an axis for $g$.  From \fullref{dynamics}, we know $\bdry \Min(g)$ is exactly the fixed point set of $\overline{g}$ viewed as a homeomorphism acting on $\bdry X$.  We also know $\bdry \Min(g)=\Sigma\bdry Y_g$ and $\bdryT \Min(g)=\Sigma(\bdryT Y_g)$ where $\{g^{\pm\infty}\}$ are the suspension points.    

Since $[N,S]$ is fixed by $\overline{g}$, we must have $[N,S]\subset\bdryT \Min(g)$.   In particular, this implies $\partial Y_g\ne\emptyset$ just as a set.  

Since $C(g)\slash\langle g\rangle$ acts geometrically on $Y_g$ and
$\partial Y_g\ne\emptyset$, $Y_g$ has 1,2, or infinitely many ends by
a theorem of Hopf which can be found on \cite[page
146]{BridsonHaefliger}.  Since $\bdry X$ is 1--dimensional, we must
have $\bdry Y_g$ 0--dimensional, thus $\bdry Y_g$ has 1, 2, or
infinitely many points.

If $Y_g$ has one end, then we would have $\Min(g)$ is isometric to a half plane.  This is impossible by \fullref{halfplane}. 

If $Y_g$ has 2 ends $\bdry \Min(g)$ a circle.  From \fullref{circle}, we must have $\Min(g)$ isometric to $\mathbb E^2\times Z$ where $Z$ is compact since $\bdryT \Min(g)$ contains $[N,S]$ (ie, is not discrete).  Thus $\bdryT \Min(g)$ is isometric to a circle as well.  

If each $h\in G$ were to leave $\Min(g)$ invariant, we would have $\bdry \Min(g)=\bdry X$ by \fullref{all}.  But then $\bdry \Min(g)$ is not a circle as assumed for this case.  Thus let $g'=hgh^{-1}$ an element of infinite order in $G$ with $\Min(g')\ne \Min(g)$.  As for $g$, we can deduce that  $\bdryT \Min(g')$ is a circle containing $[N,S]$.  In particular $\bdryT \Min(g)\cap \bdryT \Min(g')$ contains $[N,S]$.  
Since these are distinct circles, the can intersect in at most a segment of length $\pi$.  

Using \fullref{swenson2}, we have that $K=C(g)\cap C(g')$ is a convex subgroup acting on a $\cat(0)$ subset $A$ of $X$.  We also know $\bdryT A=\bdryT \Min(g)\cap\bdryT \Min(g')$ is a segment of length less than or equal to $\pi$ that contains $[N,S]$.   Since $K$ acts geometrically on $A$, $K$ must contain an element $k$ of infinite order with $C_K(k)$ acting geometrically on $\Min(k)$.  But now $\Min(k)=Y_k\times\mathbb R$ must have $\bdryT \Min(k)$ is a segment of length less than or equal to $\pi$.  Again using \fullref{halfplane}, this is impossible.  

Thus $Y_g$ has infinitely many ends and $\{g^{\pm\infty}\}=\{N,S\}$.  Indeed, $N$ and $S$ are the only points in $\bdry X$ which could possibly be the suspension points for $\bdry \Min(g)$ since all other points of $\bdry X$ are points of non-local connectivity.  But then $N$ and $S$ are the end points of a geodesic line which gives $Td(N,S)=\pi$ as needed.  
\end{proof}

\section{Group theory consequences}

The following theorem is due to M Bridson.  

\begin{thm}\label{algebra} Suppose $X$ is a proper $\cat(0)$ space and $\textrm{Isom}(X)$ has discrete orbits. If a finitely generated group $G$ acts properly discontinuously by semi-simple isometries on $X\times\mathbb R$, then there exists a subgroup $G_0$, finite index in $G$, with $G_0\cong K\times\mathbb Z$ (or just $K$) where $K$ acts properly on $X$.  
\end{thm}

\begin{proof} Any $g\in G$ acts on $X\times\mathbb R$ via $g=(g_X,t_g)$ where $g_X\in\textrm{Isom}(X)$ and $t_g\in\textrm{Isom}(\mathbb R)$.  Consider the action of $G$ on the $X$ factor given by $g\cdot x=g_X\cdot x$.  If this action is proper, then we are done by taking $G_0=G$.  If this action is not proper, then there is an infinite stabilizer $\textrm{Stab}(x)$ (since $\textrm{Isom}(X)$ has discrete orbits).  

This group $\textrm{Stab}(x)$ must act properly on $\{x\}\times\mathbb R$ since $G$ acts properly on $X\times\mathbb R$.  Since Stab$(x)$ acts properly on $L:=\{x\}\times\mathbb R$ and is
infinite, it contains a cyclic subgroup $\langle \gamma\rangle$ of finite
index where $\gamma$ acts as a translation.

Let $G^+=\langle g_1,\dots,g_l\rangle$ be the subgroup of index
at most 2 in $G$ whose action on the second factor of $X\times\mathbb R$
preserves the orientation.  

For each $g_i$ and all $p,q\in L$ we have $d(g_i(p), p) = d(g_i(q),q)$ since $L$ must go to a line parallel to $L$ under the action of any isometry.  
In particular, letting $p_n=\gamma^n(x)$,
for all $n\in\mathbb Z$ we have: 
$$
d(\gamma^{-n}g_i\gamma^n(x), x) = 
d(g_i(p_n),p_n) = d(g_i(x),x).
$$
Since the action of $G$ is proper, the set $\{\gamma^{-n}g_i\gamma^n
\mid n\in\mathbb Z\}$
must be finite. Hence, for $i=1,\dots,l$ there exists $r_i\neq 0$ such
that
$\gamma^{-r_i}g_i\gamma^{r_i}=g_i$. Thus if we let $R=r_1\dots r_l$,
then $\gamma^R$ commutes with
each of the generators of $G^+$. We can now apply \fullref{central}
to split a subgroup of finite index in $G^+$ as $K\times
\langle\gamma^R\rangle$,
where $K$ acts trivially on the the second factor of $X\times\mathbb R$
and hence
properly on the first factor.
\end{proof}

\begin{corollary}\label{geom} If $G$ acts geometrically on a $\cat(0)$ space $X\times\mathbb R$ where $\textrm{Isom}(X)$ has discrete orbits, then $G$ contains a subgroup of finite index $G_0$ of the form $K\times\mathbb Z$ where $K$ acts geometrically on $X$.
\end{corollary}

\begin{proof}  As above, consider the induced action of $G$ on the $X$ factor.  If this action is proper, then $G$ would act geometrically on $X$ as well as on $X\times\mathbb R$.  Indeed, if the $G$ translates of a compact set $C$ in $X\times\mathbb R$ cover $X\times\mathbb R$, then the $G$ translates of the projection of $C$ onto $X$ will cover $X$ under the induced $G$--action on $X$.  If $G$ acts geometrically on $X$, then we could build a geometric action of $G\times\mathbb Z$ on $X\times\mathbb R$ using the product action.  This would force $G$ to be quasi-isometric to $G\times\mathbb Z$.  This cannot happen because these two groups have different cohomlogical properties.  See the next lemma for details.  

Now by the previous theorem $G$ contains a subgroup of finite index of the form $K\times\mathbb Z$ with $K$ acting properly on $X$.  It follows that $K$ must also act cocompactly on $X$ since $G$ acts cocompactly on $X\times\mathbb R$.
\end{proof}

\begin{lemma}\label{ross} Suppose $G$ acts geometrically on a $\cat(0)$ space $X$.  Then $G$ and $G\times\mathbb Z$ cannot be quasi-isometric.  
\end{lemma}

\begin{proof}  
Since $G$ acts geometrically on $X$, we also know that $G\times\mathbb
Z$ acts geometrically on $X\times\mathbb R$ just by considering the
obvious product action.  We can apply \cite[Theorem 12]{Swenson} to
conclude that $\bdry X$ has finite Lebesgue covering dimension, call
that dimension $d$.  Using this $d$, a consequence of the main theorem
of \cite{GO} is that since $G$ acts geometrically on $X$, $H^{d+1}(G,
\mathbb ZG)\ne 0$, while $H^n(G,\mathbb ZG)=0$ for all $n > d+1$.  For
all $k$ we know $H^{k+1}(G\times\mathbb Z, \mathbb Z(G\times\mathbb
Z))\cong H^k(G,\mathbb ZG)$ for all $k$.  Thus $H^{d+2}(G\times\mathbb
Z, \mathbb Z(G\times\mathbb Z))$ is non-zero while $H^{d+2}(G,\mathbb
ZG)= 0$.  Finally, since $H^*(G,\mathbb ZG)$ is a q.-i. invariant,
\cite{Gersten}, $G$ and $G\times\mathbb Z$ cannot be quasi-isometric.
\end{proof}

\begin{thm}\label{main2}  Suppose $G$ acts geometrically on a $\cat(0)$ space $X$ with $\bdry X$ homeomorphic to the suspension of a Cantor set.  Then $G$ contains a subgroup $G_0$ of finite index which is isomorphic to $F\times\mathbb Z$ where $F$ is a nonabelian free group.
\end{thm}

\begin{proof} From \fullref{main1}, we know $X$ splits isometrically as $Y\times\mathbb R$ for some closed convex subset $Y$ of $X$ with $\bdry Y$ homeomorphic to a Cantor set and $\bdryT Y$ discrete.  By \fullref{geom}, we know that up to finite index $G$ splits as $K\times\mathbb Z$ where $K$ acts geometrically on $Y$.  Thus $Y$ is a proper, cocompact $\cat(0)$ space so we can apply the Flat Plane Theorem to $Y$.  If there were a flat plane in $Y$, then the Tits metric on $Y$ would not be discrete, and thus $Y$ must be $\delta$--negatively curved.   Since $K$ acts geometrically on $Y$, $K$ is a $\delta$--hyperbolic group.  

But now $K$ is a $\delta$--hyperbolic group acting geometrically on a
space $Y$ with Cantor set boundary.  Thus $K$ is virtually free by
work of Stallings, Gromov, Ghys and de la Harpe
\cite{Stallings,Gromov,GH}.
\end{proof}

\bibliographystyle{gtart}
\bibliography{link}

\end{document}